\documentclass[12pt]{amsart}
\newtheorem{theorem}{Theorem}[section]
\newtheorem{lemma}[theorem]{Lemma}

 \theoremstyle{definition}
\newtheorem{definition}[theorem]{Definition}

\theoremstyle{remark}
\newtheorem{remark}[theorem]{Remark}

\numberwithin{equation}{section}

\begin{document}
\setlength{\baselineskip}{1.2\baselineskip}

\title[A Liouville type theorem to $2$-Hessian equations]
{A Liouville type theorem to $2$-Hessian equations}

\author{Yan He, Haoyang Sheng, Ni Xiang}
\address{Faculty of Mathematics and Statistics, Hubei Key Laboratory of Applied Mathematics, Hubei University,  Wuhan 430062, P.R. China}
\email{helenaig@hotmail.com; 907026694@qq.com; nixiang@hubu.edu.cn}
\thanks{This research was supported by funds from Hubei Provincial Department of Education Key Projects D20171004.}

\begin{abstract}
In this paper, we proved that any 2-convex solution $u$ of $\sigma_2(D^2u)=1$
with a quadratic growth must be a quadratic polynomial in $\mathbb{R}^n\ (n\geq 3 )$ by using a Pogorelov estimate and the global gradient estimate. And we give a positive answer to the unresolved issue in \cite{CX}.


\end{abstract}

\maketitle
\bigskip

\section{Introduction}

\medskip
Let $\Omega\subset \mathbb{R}^n$ be a smooth bounded domain
and $f: \Omega\times\mathbb{R}\times\mathbb{R}^n
\rightarrow\mathbb{R}^+$
be a smooth and bounded function.
We consider
 $k$-Hessian equations,
\begin{equation}\label{01}
\sigma_k(D^2u)=f(x, u, Du),\ \textrm{in}\ \Omega,\\
\end{equation}
with the Dirichlet boundary condition,
\begin{equation}\label{02}
u=0,\ \textrm{on}\ {\partial \Omega},
\end{equation}
where $u$ is a smooth function defined in $\Omega.$
The k-th elementary symmetric polynomial
is denoted by $\sigma_k$:
\[\sigma_k(\lambda)=\sum_{1\le i_1<\cdots<i_k\le n}\lambda_{i_1}\cdots\lambda_{i_k}.\]
And $\lambda$
is the eigenvalue vector of the Hessian of $u$. A function $u\in C^2(\Omega)$ is called
to be admissible with respect to $\sigma_k$ if $\lambda(D^2 u)\in \Gamma^k$,
here $\Gamma^k$ is the Garding's cone in $\mathbb{R}^n$ determined by
\[\Gamma^k=\{\lambda=(\lambda_1,\cdots,\lambda_n)\in\mathbb{R}^n|\sigma_j(\lambda)>0, \ j=0,\cdots, k\}.\]
Obviously, $\Gamma^2$ be an open convex cone in $\mathbb{R}^n$:
\[\Gamma^2=\{\lambda=(\lambda_1,\cdots,\lambda_n)\in\mathbb{R}^n|
\sum_i\lambda_i>0,\sum_{1\le i<j\le n}\lambda_i\lambda_j>0 \}.\]

The global regularity to Dirichlet boundary problems was established by Caffarelli-Nirenberg-Spruck \cite{CNS} and Trudinger \cite{T95} in $k-1$ convex domain.
$C^2$ boundary estimates have
been improved by Bo Guan \cite{G14} under the weakened assumption that there exists an admissible sub-solution for Dirichlet problems. Guan-Ren-Wang \cite{GRW} established global estimates
resolving Weingarten
curvature equations for closed convex hypersurfaces. In the case of
scalar curvature equations $(k=2)$, they can drop the convexity
and establish the estimates for star-shaped
admissible solutions. But the general case $2 < k < n$ is still open.

The interior $C^2$ estimates were established by Warren and Yuan \cite{WY} for
\begin{equation}\label{062201}\sigma_2(D^2u)=1,\end{equation} via the minimal surface feature of the "gradient" graph $(x,Du(x))$ in dimension 3. Then, for semi-convex solutions of equations \eqref{062201} it is obtained
by McGonagle-Song-Yuan \cite{MSY} in dimension $n$. Qiu \cite{Q17} has proved the results for solutions of \eqref{01} when $k=2$ with $C^{1,1}$ variable right hand side in $\mathbb{R}^3$.
The interior estimates for convex solutions to general 2-Hessian equations \eqref{01} in $\mathbb{R}^n$ have also been obtained via a new pointwise approach by Guan and Qiu \cite{GQ}. Under
weakened conditions $\sigma_3(D^2u)\geq -A,$ where $A$ is a positive constant, interior estimates for these equations \eqref{01} have been obtained.

Since Urbas constructed counter-examples in \cite{Ur} for $k$-Hessian equations with $k \geq 3$, which generalized a result of Pogorelov \cite{P78} for Monge-Amp\`ere equations, we can expect the best is Pogorelov
type interior $C^2$ estimates. Explicitly, Chou-Wang \cite{CW}
considered \eqref{01}, when any function $f$ not depending on $D u$, and proved that any $k$ convex solution had
\begin{equation}
(-u)^{(1+\epsilon)}|D^2u |\leq C,
\end{equation}
here $\epsilon>0$. By some reasons, the small constant $\epsilon$ should not be zero in Chou-Wang's proof.
In Li-Ren-Wang  \cite{LRW}, the authors
studied Dirichlet problems \eqref{01} and \eqref{02},
proved the following estimates,
\[-u\Delta u\le C,\]
under the condition $k+1$ convex of $u$. Especially, for 2-Hessian equations, they have obtained
\[(-u)^\beta\Delta u\le C,\] here $\beta$ is sufficiently large when $u$ is 2 convex.

These type of interior estimates are important for
existence of isometric embedding of non-compact surfaces and for Liouville type theorems.
Now let us consider
equations \eqref{062201} in dimension $n$, suppose that
$u$ satisfies a\textbf{ quadratic
growth}. Here
quadratic growth means that
there exist positive constants
$b$, $c$ and
sufficiently large
$R$ such that
\begin{eqnarray}
u(x)\ge c|x|^2-b\label{1.100}
\end{eqnarray}
for $|x|\ge R$.
Our main result is stated as follows.

 \begin{theorem}\label{theorem4.1}Let~$u$~be a solution
 to the equation~(\ref{062201}).
   Suppose
 the eigenvalues of $D^2u$ lie in $\Gamma^2$
 and $u$ satisfies a quadratic growth (\ref{1.100}).
 Then $u$ is a quadratic polynomial.
\end{theorem}

\begin{remark}
In \cite{CX}, the Liouville theorem holds for equations \eqref{062201}
with the assumption $\sigma_3(D^2u)\geq -A,$ here $A$ is a positive constant. In this paper, we can relax the convexity of the solutions restriction.
\end{remark}

Let's review known results related Liouville
type theorems to $k$-Hessian equations.
If $k=n$, $k$-Hessian
equations turn to be Monge-Amp\`ere
equations
$\det D^2u=1.$ There is a well known theorem.
Jorgens \cite{J54}~proved that
entire smooth convex solution
was quadratic polynomial when $n=2$.
In 1958, Calabi \cite{C58}~proved
Liouville type theorems for
dimension $n=3, 4, 5$. Then
the result was generalized by
Pogorelov \cite{P78} to
dimension $n\ge 2$. S.Y.Cheng and S.T.Yau \cite{CY} gave another more geometric proof.
In 2003,
Caffarelli-Li, \cite{CL} extended the theorem of J$\ddot{o}$rgens, Calabi and Pogorelov based on the theory of Monge-Amp\`ere equations to viscosity solutions.

In 2003, Bao-Chen-Guan-Ji \cite{BCGJ03}~considered
Liouville type theorems to
\begin{equation}\frac{\sigma_k(D^2u)}{\sigma_l(D^2u)}=1,\ (k>l).\label{eqkl}\end{equation}
They proved that entire convex solutions
of the equation (\ref{eqkl}) with a quadratic
growth were quadratic polynomials. And in their paper, they asked if it was enough to
merely assume that u was strictly convex.
 In ~2010~, Chang-Yuan~ \cite{CY10}~
 considered Liouville type theorems for \eqref{062201},
and obtained that
 the entire solution to (\ref{062201})
 was quadratic polynomial
 if
~$$D^2u\ge\big[\delta-\sqrt{\frac{2}{n(n-1)}}\big]I$$
for $\delta>0$.
In 2016,
Li-Ren-Wang \cite{LRW}~considered
 ~$\sigma_k(D^2u)=1$ for general $k$.
They obtained that global $k+1$ convex
solutions
with a quadratic growth were quadratic
polynomials.
Recently,
Chen-Xiang \cite{CX} improved the condition
from
$(k+1)$-convex to $k$-convex for
$k=2$ under $\sigma_3(D^2u)\geq -A$. Especially, for $n=3$, $\sigma_3(D^2 u)\geq -A$ can be redundant. But for general dimension $n>3$, it can not be redundant according to the key Lemma 2.3 in \cite{CX}. Other related works include parabolic cases,
for example,
Xiong-Bao \cite{XB11},
Zhang-Bao-Wang \cite{ZBW}.

In this paper, we avoid using Lemma 2.3 in \cite{CX} and
prove a Liouville type theorem for \eqref{062201} in general dimension $n$ just with a quadratic growth, by using Pogorelov estimates in \cite{GQ} and global gradient estimates in \cite{CW}.
The paper is organized as follows.
We start with some notations
and lemmas in section 2.
The proof of a Liouville type theorem (Theorem\ref{theorem4.1})
is given in section 3.

\bigskip

\section{Preliminaries}
\label{3I-G}
\setcounter{equation}{0}

\medskip

In this section, we introduce some notations and key theorems which will be used later, and omit the details for the proof.

\begin{definition}
Let $\Omega\subset \mathbb{R}^n$ be a $C^2$ bounded domain. If $\kappa=(\kappa_1, \cdots, \kappa_n)$ represents the principal curvatures of $\partial \Omega$,
which satisfies
$$\sigma_{k-1}(\kappa)>0,\ on \ \partial \Omega,$$
then $\Omega$ is $k-1$ convex. Moreover, it is strictly $k-1$ convex with
$$\sigma_{k-1}(\kappa)\geq c_0>0,\ on \ \partial \Omega.$$
\end{definition}

\begin{definition}\label{def2}Let $\lambda=(\lambda_1,\cdots,\lambda_n)\in \mathbb{R}^n$.


(1)\[\sigma_l(\lambda|i)=\sigma_l(\lambda)\big|_{\lambda_i=0}.\]
In this paper, ~$\sigma_1(\lambda|i)$~is also denoted by~$\sigma_2^{ii}$.

(2)\[\sigma_l(\lambda|pq)=\sigma_l(\lambda)
\big|_{\lambda_p=\lambda_q=0}.\]
\end{definition}

\begin{lemma}(See  \cite{Li05})
Let~$u$~be $2$ convex.
Then ~$\sigma_2^{ij}$~is positive definite.
\end{lemma}
\begin{lemma}(See  \cite{GRW})\label{lemma1}Let~$k>l$,~$\alpha=\frac{1}{k-l}$.
For sufficiently small $\delta$,
we have
\begin{equation}
-\sigma_k^{pp,qq}u_{pph}u_{qqh}
+(1-\alpha+\frac{\alpha}{\delta})\frac{(\sigma_k)^2_h}{\sigma_k}
\ge\sigma_k(\alpha+1-\delta\alpha)
\big[\frac{(\sigma_l)_h}{\sigma_l}\big]^2
-\frac{\sigma_k}{\sigma_l}\sigma_l^{pp,qq}u_{pph}u_{qqh}.
\end{equation}

\end{lemma}

We give Pogorelov type estimates in \cite{LRW}.
\begin{theorem}(See Theorem 1 in \cite{LRW}) \label{062401}
For 2-Hessian equations with Dirichlet boundary conditions \eqref{02}, there is some large constant $\beta>0$, such that
$$(-u)^{\beta}\Delta u\leq C.$$
Here positive constants $\beta$ and $C$ denpend on the domain $\Omega$, the function $f$, $\sup_{\Omega}|u|$ and $\sup_{\Omega}|Du|$.
\end{theorem}

Global gradient estimates will be used in our proof for the Liouville theorem.

\begin{theorem}\label{theorem4.0}(See Theorem 3.4 in  \cite{CW} )
Let
$\Omega$ be a bounded smooth $k-1$ convex domain.
Suppose
$u$ is a $k$ convex solution
to the problem \eqref{01} and \eqref{02}.
Then
\[|Du|\le C,\]
where~ $C$ depends on $n, \ f,\ k, \ \partial\Omega$ and $\sup_{\Omega} |u|$.
\end{theorem}

\section{The proof of Theorem\ref{theorem4.1}}

 On the one hand, we know that $k-1$ convex of the boundary is necessary for existence of $k$-Hessian equations when $u$ is vanishing on $\partial\Omega$ in \cite{CNS}. On the other hand, for using global gradient estimates Theorem \ref{theorem4.0}, we give a simple proof for the following lemma which means equations $\eqref{062201}$ and boundary conditions \eqref{02}
imply $1$ convex property of the boundary.

\begin{lemma}\label{lemma6}Let $u$ be a solution of the
following equation.
 \begin{equation}
\left\{
\begin{array}{lr}
\sigma_2(D^2u)=1,&\textrm{in}\ \Omega,\\
u=0,&\textrm{on}\ {\partial \Omega}.
\end{array}
\right.\label{78}
\end{equation}
Then the mean curvature of $\partial\Omega$ is positive.
 \end{lemma}
\begin{proof}
Let us calculate in Fermi coordinates,
where the metric $g$
is expressed as $$g=\sum_{1\le i,j\le n-1}g_{ij}dx^{i}dx^{j}+d x^ndx^n.$$
 Let   
  $h_{ij}$ be the second
 fundamental form of $\partial \Omega$.
 Then on $\partial \Omega$, by differentiating the boundary condition twice,
 \begin{equation}u_{ii}=-h_{ii}u_n, \textrm{for} \ 1\le i\le n-1.\label{2.58}
 \end{equation}
 By Hopf lemma, we have $-u_n>0$.
 Set
 \[W=diag\{u_{11},\cdots,u_{nn}\}.\]
 Let us calculate
 at each point $x\in\partial\Omega.$
 We may assume at this point,
 \begin{eqnarray}D^2u=
 \left(
   \begin{array}{ccccc}
     u_{11} & 0      & \cdots    & 0          & u_{1n} \\
      0     & u_{22} & \cdots    & \cdots     & u_{2n} \\
     \cdots &        & \cdots    & 0          & \cdots \\
     0      &\cdots  & 0         &u_{n-1,n-1} & u_{n-1,n} \\
     u_{n1} &\cdots  & u_{n,n-2} &u_{n,n-1}   & u_{nn} \\
   \end{array}
 \right).
 \end{eqnarray}
 Therefore,
 \[\sigma_2(D^2u)=\sigma_2(W)-\sum_i^{n-1}u_{in}^2\le \sigma_2(W),\]
 \[\sigma_1(D^2u)=\sigma_1(W).\]
 Thus we have $W\in \Gamma^2$ and $\sigma_1(\lambda(D^2W)|n)>0$.
 It implies
 $-\sum_i^{n-1} h_{ii}u_n>0$
 and
 $\sum_i^{n-1}h_{ii}>0$.
\end{proof}
Now let us continue to prove
Theorem \ref{theorem4.1}:

Set~$R>1$.
Let
\[\Omega_R=\{y\in \mathbb{R}^n|u(Ry)\le R^2\},\]
and
\[v(y)=\frac{u(Ry)-R^2}{R^2}.\]
Then $v$ satisfies
Dirichlet problems
\begin{equation}\label{4.2}
\left\{
\begin{array}{lr}
\sigma_2(D^2v)=1,& in\ \Omega_R,\\
v=0,& on\ \partial \Omega_R.
\end{array}\right.
\end{equation}

By the quadratic growth condition, we obtain
\[c|Ry|^2-b\le u(Ry)\le R^2,\]
then
\[|y|^2\le \frac{b+1}{c}.\]

Note that
\[\sigma_2(D^2v)=1\le \sigma_2(D^2(|y|^2)).\]
Thus by comparison principle
we know that the minimum of~$v-|y|^2$~
is attained on the boundary.
Therefore,
\[v\ge -\frac{b+1}{c}.\]
Thus~$v$~is bounded by the absolute constants.

Moreover,
in view of Theorem \ref{theorem4.0}
and Lemma \ref{lemma6},
we have the gradient estimates.
Hence, $C$ in Theorem \ref{062401}
is an absolutely constant£º
\[-v\Delta v\le C.\]

Next, set
\[\Omega'_R=\{y|u(Ry)\le\frac{R^2}{2}\}.\]

Thus in~$\Omega'_R$~
\[v\leq-\frac{1}{2},\]
then
\[\Delta v\le C.\]
Note that
\[\Delta u=\Delta v.\]
Hence
\[\Delta u\le C\]
in
$\{x|u(x)\le\frac{R^2}{2}\}$
for any $R$.
Here $C$~is also an absolutely constant.
Lastly, using~Evans-Krylov theory (see \cite{GT}),
we have
\[\lim_{R\rightarrow\infty}|D^2u|_{C^\alpha(R)}\le\lim_{R\rightarrow\infty} C\frac{|D^2u|_{C^0(2R)}}{R^\alpha}
\le\lim_{R\rightarrow\infty}\frac{C}{R^\alpha}=0.\]
The proof is completed.


\end{document}